\documentclass[a4paper,11pt,oneside,reqno]{amsart}
\usepackage[utf8]{inputenc}
\usepackage{amsmath,amsthm,amsfonts,latexsym,amssymb,bm,enumerate}
\usepackage{ae}
\usepackage{cite}
\usepackage{float}
\usepackage{lmodern}
\usepackage[T1]{fontenc}
\usepackage[english]{babel}

\usepackage{color}

\usepackage[colorlinks=true]{hyperref}

\definecolor{dark-red}{rgb}{.54,.0,.0}
\definecolor{dark-green}{rgb}{.0,.4,.0}
\definecolor{dark-blue}{rgb}{.04,.04,.4}

\hypersetup{linkcolor=dark-red, urlcolor=dark-blue, citecolor=dark-green}

\usepackage{bbm}
\newtheorem{theorem}{Theorem} 

\newtheorem{e-proposition}[theorem]{Proposition}
\newtheorem{corollary}[theorem]{Corollary}
\newtheorem{e-definition}[theorem]{D\'efinition\rm}
\newtheorem{remark}{\it Remark\/}

\newcommand{\Rb}{{\mathbb{R}}}

\newcommand{\cal}{\mathcal}

\newcommand{\nc}{\newcommand}
\renewcommand{\phi}{\varphi}
\nc{\txt}{\textstyle}
\nc{\be}{\begin{equation}}
\nc{\ee}{\end{equation}}
\nc{\ba}{\begin{eqnarray}}
\nc{\ea}{\end{eqnarray}}
\nc{\bas}{\begin{eqnarray*}}
\nc{\eas}{\end{eqnarray*}}
\nc{\weak}{\rightharpoonup}
\nc{\paq}{$({\cal{P}})_{\alpha,q}$}
\nc{\paz}{$({\cal{P}})_{\alpha,\tz}$}
\nc{\Om}{\Omega}
\nc{\ek}{{\eps_{k}}}
\nc{\ak}{{\alpha_{k}}}
\nc{\ck}{{C_{k}}}
\nc{\ct}{\tilde{C}}
\nc{\cc}{\hat{C}}
\nc{\uk}{{u_{k}}}
\nc{\Uk}{{U_{k}}}
\nc{\ts}{{2^*}}
\nc{\tsu}{{{2^*}-1}}
\nc{\tsd}{{{2^*}-2}}
\nc{\tst}{{{2^*}-3}}
\nc{\tz}{{2^\#}}
\nc{\tzu}{{{2^\#}-1}}
\nc{\tzd}{{{2^\#}-2}}
\nc{\wk}{{w_{k}}}
\nc{\sndd}{\frac{S^\frac N2}2}
\nc{\sddn}{\frac{S}{2^{\frac 2N}}}
\nc{\igukw}{{\int\nabla\Uk\cdot\nabla\wk}}
\nc{\iukw}{{\int U_k\wk}}
\nc{\iuks}{{|U_k|_{2^*}^{2^*}}}
\nc{\iukz}{{|U_k|_{2^\#}^{2^\#}}}
\nc{\iukpz}{{|u_k|_{2^\#}^{2^\#}}}
\nc{\iusu}{{\int\ U_k^{{2^*}-1} w_k}}
\nc{\iusd}{{\int\ U_k^{{2^*}-2} w_k^2}}
\nc{\iust}{{\int\ U_k^{{2^*}-3} w_k^3}}
\nc{\iuzu}{{\int\ U_k^{{2^\#}-1} |w_k|}}
\nc{\iuzd}{{\int\ U_k^{{2^\#}-2} w_k^2}}
\nc{\iuz}{{|u|_{\tz}^\tz}}
\nc{\ius}{{|u|_{\ts}^\ts}}
\nc{\nguk}{{|\nabla\Uk|_{2}}}
\nc{\ngukq}{{|\nabla\Uk|_{2}^2}}
\nc{\ngwkq}{{|\nabla\wk|_{2}^2}}
\nc{\nuk}{{|\Uk|_{2}}}
\nc{\nukq}{{|\Uk|_{2}^2}}
\nc{\nuks}{{|\Uk|_{\ts}}}
\nc{\nuksq}{{|\Uk|_{\ts}^2}}
\nc{\nukss}{{|\Uk|_{\ts}^\ts}}
\nc{\nuksmd}{{|\Uk|_{\ts}^{-2}}}
\nc{\nwk}{{||\wk||}}
\nc{\nwks}{{|\wk|_{\ts}}}
\nc{\nwksq}{{|\wk|_{\ts}^2}}
\nc{\ndwkq}{{|\wk|_{2}^2}}
\nc{\nwkq}{{||\wk||^2}}
\nc{\nwkr}{{||\wk||^r}}
\nc{\ngukpq}{{|\nabla u_{k}|_{2}^2}}
\nc{\ngupq}{{|\nabla u|_{2}^2}}
\nc{\nukps}{{|u_{k}|_{\ts}}}
\nc{\nukpss}{{|u_{k}|_{\ts}^\ts}}
\nc{\nups}{{|u|_{\ts}}}
\nc{\nupss}{{|u|_{\ts}^\ts}} 
\nc{\nupsz}{{|u|_{\ts}^\tz}} 
\nc{\nukpsq}{{|u_{k}|_{\ts}^2}}
\nc{\nupsq}{{|u|_{\ts}^2}}
\nc{\nupq}{{|u|_{2}^2}}
\nc{\nukpq}{{|u_k|_{2}^2}}
\nc{\ql}{\frac{\ngukq}{|\Uk|_{\ts}^{2^*}}}
\nc{\beuk}{\frac{\ngukq+a\nukq}{\nuksq}}
\nc{\nhu}{||u||^2}
\nc{\pnhu}{||u||^2}
\nc{\pnhumeio}{||u||}
\nc{\nhukpq}{||u_{k}||^2}
\nc{\pnhukpq}{||u_{k}||^2}
\nc{\pnhukpqmeio}{||u_{k}||}
\nc{\nhukq}{||U_{k}||^2}
\nc{\pnhukq}{||U_{k}||^2}
\nc{\nhwkq}{||w_{k}||^2}
\nc{\pnhwkq}{||w_{k}||^2}
\nc{\ttt}{{\Rb^{N}_{+}}}
\nc{\ue}{U_{\eps}}
\nc{\pb}{\bar{\phi}_{\eps}}
\nc{\pbq}{\bar{\phi}_{\eps}^2}
\nc{\pk}{{\phi}_{\eps}}
\nc{\uz}{U^\tzd_{\eps}}
\nc{\uzz}{U^\tzd_{\eps,y_{\eps}}}
\nc{\us}{U^\tsd_{\eps}}
\nc{\uss}{U^\tsd_{\eps,y_{\eps}}}
\nc{\et}{\eta_{\eps}}
\nc{\etq}{\eta_{\eps}^2}
\nc{\pt}{{\tilde\phi}}
\nc{\pte}{{\tilde{\phi}_{\eps}}}
\nc{\ptq}{{{\tilde\phi}^2}}
\nc{\pteq}{\tilde{\phi}_{\eps}^2}
\nc{\ome}{{\Om_{\eps}}}
\nc{\mk}{M_{k}}
\nc{\dk}{\delta_{k}}
\nc{\pkk}{P_{k}}
\nc{\ndd}{{\frac{N-2}{2}}}
\nc{\ium}{U_{\delta_{k},P_{k}}}
\nc{\vk}{{v_{k}}}
\nc{\gb}{\gamma+\sqrt{\gamma^2+4\beta}}
\nc{\gbb}{\gamma_{\infty}+\sqrt{\gamma_{\infty}^2+4\beta_{\infty}}}
\nc{\ia}{\frac{1}{4(\tz)^{\frac{2}{N}}}
         \left[\left(\gb\right)^{N}+2\cdot\ts\beta\left(\gb\right)^{N-2}
         \right]^{\frac{2}{N}}}
\nc{\iab}{\frac{1}{4(\tz)^{\frac{2}{N}}}
         \left[\left(\gbb\right)^{N}+2\cdot\ts\beta_{\infty}\left(\gbb\right)^{N-2}
         \right]^{\frac{2}{N}}}
\nc{\aaa}{\mbox{{\sl a}}}
\nc{\bg}{\mbox{\b{$\gamma$}}}
\nc{\bb}{\mbox{\b{$\beta$}}}
\nc{\zs}{{\frac{\tz}{\ts}}}
\nc{\sz}{{\frac{\ts}{\tz}}}
\nc{\ds}{{\frac{2}{\ts}}}
\nc{\gx}{\bg x^{\zs}}
\nc{\sd}{S_{2}}
\nc{\gxq}{\bg^2 x^{2\zs}}
\nc{\az}{S^{\frac N2}\frac{A(N)}{B(N)}\max_{\partial\Om}H}
\nc{\azz}{A(N)\max_{\partial\Om}H}
\nc{\bd}{\mbox{\b{$\delta$}}}
\nc{\dd}{\frac{1}{2}\alpha\delta}
\nc{\tzud}{\frac{\tz+1}{2}}
\nc{\tzdois}{\frac{\tz}{2}}
\nc{\lb}{\left(}
\nc{\rb}{\right)}
\nc{\defdel}{\frac{\iuz}{\pnhumeio\cdot|u|_{\ts}^{\ts\!/2}}}
\nc{\intum}{\int(\nabla u\cdot\nabla\phi+au\phi)}
\nc{\intdois}{\int u^\tzu\phi}
\nc{\inttres}{\int u^{\tsu}\phi}
\nc{\cs}{{\cal{S}}}
\nc{\sddnd}{S/2^{\frac{2}{N}}}   
\newcounter{um}

\newcommand{\eps}{\varepsilon}

\begin{document}

\title{A sharp inequality for Sobolev functions}

\author{Pedro M.\ Gir\~{a}o}

\address{Mathematics Department, Instituto Superior T\'{e}cnico, Av.\ 
    Rovisco Pais, 1049-001 Lisbon, Portugal}
\email{pgirao@math.ist.utl.pt}

\begin{abstract}
Let $N\geq 5$, $a>0$, 
        $\Om$ be a smooth bounded domain in $\Rb^{N}$,
$\ts=
\frac{2N}{N-2}$, 
$\tz=\frac{2(N-1)}{N-2}$ and $||u||^2=\ngupq+a\nupq$.  We prove there exists an 
$\alpha_{0}>0$ such that, for all $u\in 
H^1(\Om)\setminus\{0\}$,
    $$\sddn\leq\frac\nhu\nupsq\lb1+\alpha_{0}
    \frac{\iuz}{\pnhumeio\cdot|u|_{\ts}^{\ts\!/2}}\rb.$$ 
    This inequality implies Cherrier's inequality. 
\end{abstract}

\maketitle

%\vspace{3mm}

\noindent Let $N\geq 5$, $a>0$, $\alpha\geq 0$, $\Om$ be a smooth 
bounded domain in $\Rb^{N}$,
$\ts=
\frac{2N}{N-2}$ 
and $\tz=\frac{2(N-1)}{N-2}$. 
We regard $a$ as fixed and 
$\alpha$ as a parameter.
Denote the $L^p$ and $H^1$ norms of $u$ in $\Om$ by
$$|u|_{p}:=\left(\textstyle\int|u|^p\right)^\frac{1}{p}\qquad  
\mbox{and}\qquad
||u||:=\left(\ngupq+a\nupq\right)^\frac{1}{2}\!\!,$$
respectively.
All our integrals are over $\Om$. We define the 
functional $\delta:H^1(\Om)\setminus\{0\}\to\Rb$, 
homogeneous of degree zero,
by
$$
\delta(u):=
\frac{\iuz}{\pnhumeio\cdot|u|_{\ts}^{\ts\!/2}}
$$
and consider the system
\addtocounter{equation}{+1}
\setcounter{um}{\theequation}
\nc{\rf}{$(\theum)_{\alpha}$}
\nc{\rfk}{$(\theum)_{\alpha_{k}}$}
\nc{\rfz}{$(\theum)_{\alpha_{0}}$}
$$
\left\{\begin{array}{ll}
\lb 1+\dd(u)\rb(-\Delta u+au)+\tzdois\alpha u^\tzu=\lb 1+
\tzud\alpha\delta(u)\rb u^\tsu&\mbox{in\ }\Om,\\
u>0&\mbox{in\ }\Om,\\
\frac{\partial u}{\partial\nu}=0&\!\!\!\!\mbox{on\ }\partial\Om.
\end{array}\right.\eqno{(\theequation)_{\alpha}}
$$

We claim that the solutions of \rf\ correspond to critical points of 
the functional $\Phi_{\alpha}:H^1(\Om)\setminus\{0\}\to\Rb$, defined by
\be\label{phi}
\Phi_{\alpha}(u):=\lb
\frac 12\nhu-\frac{1}{\ts}\ius
\rb\lb1+\alpha\defdel\rb^{\frac{N}{2}}=
\Phi_{0}(u)(1+\alpha\delta(u))^{\frac{N}{2}}.
\ee
In fact, since 
$$
\Phi_{\alpha}^\prime=(1+\alpha\delta)^{\frac{N}{2}-1}\left[
\Phi_{0}^\prime(1+\alpha\delta)+{\textstyle \frac{N}{2}}
\Phi_{0}\alpha\delta'
\right]
$$
and, for $\phi\in H^1(\Om)$, 
\begin{eqnarray*}
\delta'(u)(\phi)&=&
-
\frac{\delta(u)}{\nhu}\intum
+\tz\frac{\delta(u)}{\iuz}
\int (|u|^\tzd u\phi)\\
&&-\frac{\ts}{2\,}
\frac{\delta(u)}{\ius}\int (|u|^{\tsd}u\phi),
\end{eqnarray*}
the critical points of $\Phi_{\alpha}$ satisfy
$$\begin{array}{rcl}\textstyle
(-\Delta u+au)\lb 1+\frac{4-N}{4}\alpha\delta(u)+\frac{N-2}{4}
\frac{\ius}{\nhu}\alpha\delta(u)\rb &&\\
+\
\frac{\tz N}{2}\alpha |u|^\tzd u\lb\frac 12\frac{||u||}{|u|_{\ts}^{\ts\!/2}}-
\frac{1}{\ts}\frac{|u|_{\ts}^{\ts\!/2}}{||u||}\rb&&\\
-\ |u|^\tsd u\lb 1+\frac{4-N}{4}\alpha\delta(u)+\frac{\ts 
N}{8}\frac{\nhu}{\ius}\alpha\delta(u)\rb&=&0
\end{array}
$$
in $\Om$, $\lb\frac{\partial u}{\partial\nu}=0\mbox{\ on\ 
}\partial\Om\rb$.
However, multiplying this equation by $u$ and integrating over $\Om$
(i.e.\ differentiating (\ref{phi}) along the radial 
direction)
we get $||u||^2=\ius$.  Conversely, the solutions of \rf\ are 
solutions of the previous equation\!\!: multiplying \rf\ by $u$ and integrating 
over $\Om$ we get a quadratic equation in $||u||/|u|_{\ts}^{\ts\!/2}$, whose 
solution is $||u||/|u|_{\ts}^{\ts\!/2}=1$.
This proves our claim.

The functional $\Phi_{\alpha}$ restricted to the Nehari manifold,
$${\cal{N}}:=\left\{u\in H^1(\Om)\setminus\{0\}:
\Phi_{\alpha}^\prime(u)u=0\right\}
=\left\{u\in H^1(\Om)\setminus\{0\}:||u||^2=\ius\right\},$$
is
$
\frac 1N[\beta(1+\alpha\delta)]^\frac{N}{2},
$
where $\beta:H^1(\Om)\setminus\{0\}\to\Rb$ is defined by
$$
\beta(u):=\frac\nhu\nupsq.
$$
So, we consider the functional 
$\Psi_{\alpha}:H^1(\Om)\setminus\{0\}\to\Rb$,
defined by
$$\Psi_{\alpha}:=\beta(1+\alpha\delta).
$$
A {\em 
least energy\/} solution of 
\rf\ is a function $u\in H^1(\Om)\setminus\{0\}$, such that
$$
\Phi_{\alpha}(u)=\inf_{{\cal{N}}}\Phi_{\alpha}=\inf_{H^1(\Om)\setminus\{0\}}
{\textstyle\frac{1}{N}} (\Psi_{\alpha})^\frac{N}{2}.
$$
We are interested in proving existence and nonexistence of least 
energy solutions of \rf.  We note that
every critical point of $\Phi_{\alpha}$ is a 
critical point of $\Psi_{\alpha}$.  
It is easy to check that the Nehari manifold is a 
natural constraint for $\Phi_{\alpha}$.  So
conversely, if $u$ is a critical point of
$\Psi_{\alpha}$, then 
there exists a unique $t(u)>0$, such that 
$t(u)u$ is a critical point of
$\Phi_{\alpha}$ $\left( 
t(u)=\left(\nhu/|u|_{\ts}^{\ts}\rb^\frac{N-2}{4}\right)$.

We consider the minimization 
problem corresponding to
$$
S_{\alpha}:=\inf\left\{\Psi_{\alpha}(u)|u\in 
H^1(\Om)\setminus\{0\}\right\}.
$$
We recall that 
$
S:=\inf\left\{\left.\frac{|\nabla u|_{L^2(\Rb^N)}^2}{|u|_{L^{\ts}\!(\Rb^N)}^2}\right|u\in
L^\ts(\Rb^{N}), \nabla u\in L^2(\Rb^{N}),
u\neq 0\right\}
$
is achieved by the instanton
$
U(x):=\left(\frac{N(N-2)}{N(N-2)+|x|^2}\right)^{\frac{N-2}{2}}
$. 
Our main result is
\begin{theorem}\label{theorem}
    There exists a positive real number
$$\alpha_{0}=\alpha_{0}(a,\Om)=
\min\left\{\alpha\,|\;S_{\alpha}=S/2^\frac{2}{N}\right\}$$ such that
\begin{enumerate}
    \item[(i)] if $\alpha<\alpha_{0}$, then
    \rf\ has a least energy solution $u_{\alpha}$;
    \item[(ii)] if 
    $\alpha>\alpha_{0}$, then \rf\ does not have a least energy 
    solution and
    $
    \Psi_{\alpha}\geq S/2^\frac{2}{N}.
    $
The constant $ S/2^\frac{2}{N}$ is sharp.
\end{enumerate}
\end{theorem}
\begin{remark}
Obviously, $\alpha_{0}$ is a nonincreasing function of $a$.
By testing $\Psi_{\alpha}$ with constant functions and instantons we 
can prove that
$$
\alpha_{0}\geq\max\left\{
[{S/(2|\Om|)^{\frac{2}{N}}-a}]/{\sqrt{a}}, C(N)\max_{\partial\Om}H
\right\}
$$
where $|\Om|$ is the Lebesgue measure of $\Om$, $H$ is the mean 
curvature of $\partial\Om$ and $C(N)$ is a constant that only depends 
on $N$.  The least energy solutions might be constant 
($a^\frac{N-2}{4}$) for $a\leq S/(2|\Om|)^{\frac{2}{N}}$ if
$\alpha\leq [{S/(2|\Om|)^{\frac{2}{N}}-a}]/{\sqrt{a}}$.
\end{remark}
\begin{corollary} For all $u\in 
H^1(\Om)\setminus\{0\}$,
    $$\sddn\leq\frac\nhu\nupsq\lb1+\alpha_{0}
    \frac{\iuz}{\pnhumeio\cdot|u|_{\ts}^{\ts\!/2}}\rb.
    $$
\end{corollary}
\begin{corollary} For all $u\in H^1(\Om)$,
    $$\sddn\nupsq\leq\nhu+\alpha_{0}||u||\cdot|u|_{2}
    \qquad\mbox{and}\qquad\sddn\nupsq\leq
    \lb|\nabla u|_{2}+c_{a,\alpha_{0}}|u|_{2}\rb^2,
    $$
    with 
    $c_{a,\alpha_{0}}=\max\left\{\alpha_{0}/2,
    \sqrt{a+\alpha_{0}\sqrt{a}\,\,}\right\}$.
\end{corollary}
\begin{proof}
From H\"older's inequality
$|u|_{\tz}^\tz\leq |u|_{2}|u|_{\ts}^{{\ts\!/2}}$, so
$\delta(u)\leq\frac{|u|_{2}}{||u||}$.  
\end{proof}
\begin{corollary} (Cherrier's inequality). Let $\eps>0$.  For all $u\in H^1(\Om)$,
    $$\sddn\nupsq\leq(1+\eps)\nhu+\frac{\alpha_{0}^2}{4\eps}|u|_{2}^2=
    (1+\eps)\ngupq+\lb\frac{\alpha_{0}^2}{4\eps}+a\eps\rb|u|_{2}^2.$$
\end{corollary}

{\em Sketch of the proof of}\/ Theorem~\ref{theorem}. $-$
By testing $\Psi_{\alpha}$ with instantons,
$S_{\alpha}\leq\sddn$, for all $\alpha\geq 0$.
We claim that if $S_{\alpha}<\sddn$, then $S_{\alpha}$ is achieved. 
This is a consequence of the concentration-compactness principle.  A 
minimizing sequence $\uk$ with $|\uk|_{\ts}=1$ is bounded and we can 
assume $\uk\weak u$ in $H^1(\Om)$,
$\lim_{k\to\infty}\ngukpq=\ngupq+||\mu||$
and $\lim_{k\to\infty}\nukpss=\nupss+||\nu||=1$,
where $\sddn||\nu||^\frac 2\ts\leq
||\mu||$ (we remark that this inequality follows from 
Cherrier's inequality).
We can write
$\beta(u)\delta(u)=\gamma(u)\sqrt{\beta(u)}$ for
$\gamma(u)={\iuz}/{\nupsz}$.  The key step 
in the proof of the claim is 
the following observation.
Define $f$ and $g:[0,1]\to\Rb$, by
$$\begin{array}{lcl}
f(x)&:=&\textstyle
\beta x^\frac{2}{\ts}+\sddn(1-x)^\frac{2}{\ts}+\alpha\gamma 
x^\frac{\tz}{\ts}
\sqrt{\beta x^\frac{2}{\ts}+\sddn(1-x)^\frac{2}{\ts}}\\
&\geq& \textstyle
\beta x+\sddn(1-x)+\alpha\gamma x\sqrt{\beta x+\sddn(1-x)}
\ \ =:\ \ g(x).
\end{array}
$$
Suppose $\min f<\sddn$.  It follows that $\beta<\sddn$, and this in turn 
implies that 
$g$ is concave.  Since $f$ and $g$ coincide at 0 and 1, the minimum of 
$f$ occurs at 1.  This proves the claim.

A similar argument
shows
$\alpha\mapsto S_{\alpha}$
is continuous.  In particular, the supremum 
$\alpha_{0}:=\sup\{\alpha|S_{\alpha}<S/2^\frac 2N\}$ is either $+\infty$ or a 
maximum.
The map $\alpha\mapsto S_{\alpha}$ is strictly increasing on 
$[0,\alpha_{0}]$.
If $\alpha\in[0,\alpha_{0}[$, then \rf\ has a least energy solution. 
    If $\alpha\in]\alpha_{0},+\infty[$, then \rf\ does 
    not have a least energy solution.  It remains to prove that 
    $\alpha_{0}$ is finite.  
    
    Suppose 
    $\alpha_{0}=+\infty$.  Choose $\alpha_{k}\to+\infty$ as $k\to+\infty$ and
    let $\uk$ be minimizer for $S_{\alpha_{k}}$ satisfying 
    \rfk.
    It is easy to prove that 
    $\lim_{\alpha\to\infty}S_{\alpha}=S/2^\frac 2N$, 
    $\mk:=\max_{\bar\Om}\uk=\uk(P_{k})\to+\infty$,
    $\lim_{k\to\infty}\ngukpq=\lim_{k\to\infty}\nukpss=S^\frac 
    N2/2$
    and $\ak\delta(\uk)\to 0$.  We can apply the Gidas-Spruck blow up 
    technique to \rfk\ because $\ak\delta(\uk)\to 0$.  
    Define $
    \eps_{k}:=M_{k}^{-{2}/(N-2)}$ and
      $U_{\eps,y}:=\eps^{-\frac{N-2}{2}}U\left(\frac{x-y}{\eps}\right)$.
    We can prove that
    $
    \lim_{k\to\infty}\ak\eps_{k}=0$,
    $
    \lim_{k\to\infty}|\nabla\uk-\nabla U_{\eps_{k},P_{k}}|_{2}=0
    $
    and $\pkk\in\partial\Om$, for large $k$.
        
At this point, using the ideas of \cite{APY}, we follow the argument in 
\cite{CG}, which applies with no modification.  We 
show 
$\Psi_{\ak}(\uk)>\sddnd$, for large $k$.  This is impossible.  
Therefore $\alpha_{0}$ is finite.\hfill$\Box$

\begin{remark}
The functional behavior
behind this inequality is also present, for example, in 
the Dirichlet problem for $-\Delta u-au+\alpha u^{1/3}=u^{7/3}$
in $\Om\subset\Rb^5$, with
$0<a<\lambda_{1}\left(-\Delta,H^1_{0}(\Om)\right)$. In this case 
$s:=[t(u)]^{2/3}$ is the
solution of a cubic equation,
$(|\nabla u|_{2}^2-a|u|_{2}^2)s+\alpha|u|_{4/3}^{4/3}-|u|_{10/3}^{10/3}s^3=0$.
\end{remark}

\end{document}